\documentclass{amsart}

\usepackage[left=1.5in, right=1.5in]{geometry}   
\usepackage{graphicx}              
\usepackage{amsmath}               
\usepackage{amsfonts}              
\usepackage{amsthm}                
\usepackage{enumitem}
\usepackage{amssymb}
\usepackage{url}
\usepackage{color}
\usepackage{mathrsfs}
\usepackage{mathtools}
\usepackage{mathabx}
\usepackage{dsfont}

\numberwithin{equation}{section}

\newtheorem{thm}{Theorem}[section]
\newtheorem{lem}[thm]{Lemma}
\newtheorem{prop}[thm]{Proposition}
\newtheorem{cor}[thm]{Corollary}

\theoremstyle{definition}
\theoremstyle{definition}
\theoremstyle{definition}\newtheorem*{remark}{Remark}

\newcommand{\norm}[1]{\left\lVert#1\right\rVert} 
\newcommand{\RR}{\mathbb{R}}            

\newcommand{\CC}{\mathbb{C}}
\newcommand{\ZZ}{\mathbb{Z}}

\newcommand{\Intr}{\displaystyle\int}
\newcommand{\Sumn}{\displaystyle\sum}
\newcommand{\Suma}{\Sumn_{\alpha\in R_+}}
\newcommand*\diff{\mathop{}\!\mathrm{d}}

\newcommand{\al}{\alpha}
\newcommand{\alx}{\langle\alpha,x\rangle}

\newcommand{\DT}{\mathcal{D}_k}
\newcommand{\IntN}{\Intr_{\RR^N}}
\newcommand{\1}{\mathds{1}}
\newcommand\restr[2]{{
  \left.\kern-\nulldelimiterspace  #1 
  \vphantom{\big|} 
  \right|_{#2} 
  }}

\begin{document}

\title{Sobolev-Type Inequalities for Dunkl Operators}
\author{Andrei Velicu}
\address{Andrei Velicu, Department of Mathematics, Imperial College London, Huxley Building, 180 Queen's Gate, London SW7 2AZ, UK}
\email{a.velicu15@imperial.ac.uk}
\date{}

\begin{abstract}
In this paper we study the Sobolev inequality in the Dunkl setting using two new approaches which provide a simpler elementary proof of the classical case $p=2$, as well as an extension to the coefficient $p=1$ that was previously unknown. We also find estimates of the sharp constants for the Sobolev inequality for Dunkl gradient. Related inequalities and some improvements are also considered (Nash inequality, Besov space embeddings).
\end{abstract}

\maketitle

\section{Introduction}

The classical Sobolev inequality states that
\begin{equation} \label{classicalsobolev}
\norm{f}_q \leq C \norm{\nabla f}_p \qquad \forall f\in C_c^\infty (\RR^N) 
\end{equation}
where $1\leq p <N$ and $q=\frac{Np}{N-p}$. This is a fundamental result in analysis and it has been widely studied in a variety of contexts, see e.g. the classical references \cite{Mazya}, \cite{SC}. 

We will be concerned with this inequality (and related inequalities) in the context of Dunkl theory. The general approach to inequality (\ref{classicalsobolev}) is to represent $f$ as an integral expression involving $\nabla f$. In the Dunkl setting, this type of representation is given by the Riesz transform and Riesz potential, and indeed the Sobolev inequality  for $1<p<N+2\gamma$ was obtaind as a corollary of this theory in \cite{AS} (see also \cite{HMS}).

Combining our results with the existing results of \cite{AS}, we have obtained the Sobolev inequality for Dunkl operators in full generality. That is, we have the following Theorem.

\begin{thm} (Sobolev inequality)
Let $1 \leq p < N+2\gamma$ and $q=\frac{p(N+2\gamma)}{N+2\gamma-p}$. Then there exists a constant $C>0$ such that we have the inequality 
$$ \norm{f}_q \leq C \norm{\nabla_k f}_p \qquad \forall f\in C_c^\infty(\RR^N).$$
\end{thm}

In this paper we explore different approaches that provide simpler proofs, or which improve on the existing results. We will first prove the Nash inequality in the Dunkl setting and use this to obtain an elementary proof of the Sobolev inequality in the classical case $p=2$. Nash's inequality is another important result in analysis and it was first proved in \cite{N} where it was used in the study of parabolic and elliptic equations. Nash's inequality can be seen as a weaker version of the Sobolev inequality as it can be deduced from the latter using only H\"older's inequality, but the two are in fact equivalent. We will prove this in our context using a nice elementary result of \cite{BCLS}. 

Using a different approach, based on a pseudo-Poincar\'e inequality and a method of Ledoux \cite{L}, we will prove a more general Besov space result.  This implies in particular  the Sobolev inequality in the case $1\leq p \leq 2$ (the limitation $1\leq p \leq 2$ comes from the pseudo-Poincar\'e inequality). Note that this includes the case $p=1$ which was not known before. A Gagliardo-Nirenberg inequality is also obtained.

Finally, we consider the problem of estimating the best constant in the Sobolev inequality. In the classical case (inequality (\ref{classicalsobolev})), this amounts to finding the supremum
$$ C = \displaystyle\sup_{f\in C_c^1(\RR^N)} \frac{\norm{f}_q}{\norm{\nabla f}_p},$$
and the question was answered in \cite{T} and \cite{A}. In the paper of Talenti, it was shown that replacing $f$ by its symmetric decreasing rearrangement $f^*$ increases the fraction in the definition of $C$, so it is enough to consider the supremum over radial functions. This was done using using the P\'olya-Szeg\H{o} inequality 
\begin{equation} \label{polyaszegoclassical}
\norm{\nabla f^*}_p \leq \norm{\nabla f}_p.
\end{equation}
which holds for all $p>1$. This simplifies the problem to that of  maximising a functional over a space of functions defined on the real positive half-line. 

Throughout this paper we use a simple but very useful result based on the carr\'e-du-champ operator which provides a link between the $L^2$ norms of the usual gradient $\nabla f$, and the Dunkl gradient $\nabla_k f$. Namely, we have
$$ \norm{\nabla f}_2 \leq \norm{\nabla_k f}_2.$$
This is particularly useful in the present paper in estimating the best constant in the Dunkl Sobolev inequality by linking it to its weighted counterpart involving only the usual gradient. 

We deal with this question in Section 6. First we prove an isoperimetric inequality for the weighted Dunkl measure $\mu_k$, and using this we can appeal to a general result of Talenti \cite{T2}, who proves a weighted version of the P\'olya-Szeg\H{o} inequality. This provides us with a weighted rearrangement inequality that holds for usual gradient $\nabla f$, on Weyl chambers. In particular, we find precise best constant for the weighted Sobolev inequality in this classical case. We will then connect this inequality to Dunkl operators by exploiting properties of Dirichlet forms, as discussed above.  

The structure of the paper is as follows. In Section 2 we present a brief introduction to the classical theory of Dunkl operators. Section 3 introduces the carr\'e-du-champ operator and we also prove some related results that will be very important in all subsequent proofs. In section 4 we prove Nash's inequality and then deduce the Sobolev inequality in case $p=2$. In Section 5 we prove the pseudo-Poincar\'e inequality and the Besov space inequality. Finally, in Section 6 we prove the isoperimetric inequality and the rearrangement inequality, and find exact best constants for the weighted Sobolev inequality for usual partial derivatives. This is then used to deduce important results about best constants for the Sobolev inequality in the Dunkl setting. 

\section{Preliminaries}

In this section we will present a very quick introduction to Dunkl operators. For more details see \cite{H} for the general theory of root systems, and the survey papers \cite{Rosler} and \cite{Anker} for an overview of Dunkl theory.

A root system is a finite set $R\subset \RR^N\setminus \{0\}$ such that $R \cap \alpha \RR = \{ -\alpha, \alpha\}$ and $\sigma_\alpha(R) = R$ for all $\alpha\in R$. Here $\sigma_\alpha$ is the reflection in the hyperplane orthogonal to the root $\alpha$, i.e.,
$$ \sigma_\alpha x = x - 2 \frac{\alx}{\langle \alpha,\alpha \rangle} \alpha.$$
The group generated by all the reflections $\sigma_\alpha$ for $\alpha\in R$ is a finite group, and we denote it by $G$. 

Let $E$ be the set of all functions $\epsilon: R_+ \to \{ -1,1\}$, and for each $\epsilon\in E$ let 
$$ \RR^N_\epsilon = \{ x\in\RR^N : \text{sgn}(\alx)=\epsilon(\alpha) \text{ for all } \alpha\in R_+\}.$$
The Weyl chambers associated to the root system $R$ are the connected components of $\RR^N \setminus \{x\in\RR^N : \alx =0 \text{ for some } \alpha\in R \}$. Weyl chambers are all of the form $\RR^N_\epsilon$ for some $\epsilon\in E$; equivalently, if $\RR^N_\epsilon \neq \{ 0 \}$, then $\RR^N_\epsilon$ is a Weyl chamber. It can be checked that the reflection group $G$ acts simply transitively on the set of Weyl chambers so, in particular, the number of Weyl chambers equals the order of the group, $|G|$.

Let $k:R \to [0,\infty)$ be a $G$-invariant function, i.e., $k(\alpha)=k(g\alpha)$ for all $g\in G$ and all $\alpha\in R$. We will normally write $k_\alpha=k(\alpha)$ as these will be the coefficients in our Dunkl operators. We can write the root system $R$ as a disjoint union $R=R_+\cup (-R_+)$, and we call $R_+$ a positive subsystem;  this decomposition is not unique, but the particular choice of positive subsystem does not make a difference in the definitions below because of the $G$-invariance of the coefficients $k$.

From now on we fix a root system in $\RR^N$ with positive subsystem $R_+$. We also assume without loss of generality that $|\alpha|^2=2$ for all $\al
\in R$. For $i=1,\ldots, N$ we define the Dunkl operator on $C^1(\RR^N)$ by
$$ T_i f(x) = \partial_i f(x) + \Suma k_\alpha \alpha_i \frac{f(x)-f(\sigma_\alpha x)}{\alx}.$$
We will denote by $\nabla_k=(T_1,\ldots, T_N)$ the Dunkl gradient, and $\Delta_k = \displaystyle\sum_{i=1}^N T_i^2$ will denote the Dunkl laplacian. Note that for $k=0$ Dunkl operators reduce to partial derivatives, and $\nabla_0=\nabla$ and $\Delta_0=\Delta$ are the usual gradient and laplacian.

We can express the Dunkl laplacian in terms of the usual gradient and laplacian using the following formula:
\begin{equation} \label{Dunkllaplacian}
\Delta_k f(x) = \Delta f(x) + 2\Suma k_\alpha \left[ \frac{\langle \nabla f(x),\alpha \rangle}{\alx} - \frac{f(x)-f(\sigma_\alpha x)}{\alx^2} \right].
\end{equation}

The weight function naturally associated to Dunkl operators is
$$ w_k(x) = \prod_{\alpha\in R_+} |\alx|^{2k_\alpha}.$$
This is a homogeneous function of degree
$$ \gamma := \Suma k_\alpha.$$
We will work in spaces $L^p(\mu_k)$, where $\diff\mu_k = w_k(x) \diff x$ is the weighted measure; the norm of these spaces will be written simply $\norm{\cdot}_p$. With respect to this weighted measure we have the integration by parts formula
$$ \IntN T_i(f) g \diff\mu_k = - \IntN f T_i(g) \diff\mu_k.$$

The Macdonald-Mehta integral associated to the root system $R$ is defined as
$$ M_k = \IntN e^{-|x|^2/2} \diff\mu_k(x).$$ 
Macdonald conjectured in \cite{M} that 
\begin{equation} \label{macdonald-mehta} 
M_k = (2\pi)^{N/2} \displaystyle\prod_{\alpha\in R_+} \frac{\Gamma(2k_\alpha+1)}{\Gamma(k_\alpha+1)}.
\end{equation}
He also proved the result for infinite classes of root systems. Later, Opdam proved in \cite{O} the result in general for all crystallographic root systems. Finally, a proof in full generality was given by Etingof in \cite{E}.

An important function associated with Dunkl operators is the Dunkl kernel $E_k(x,y)$, defined on $\CC^N\times\CC^N$, which acts as a generalisation of the exponential and is defined, for fixed $y\in\CC^N$, as the unique solution $Y=E_k(\cdot, y)$ of the equations
$$ T_iY=y_i Y, \qquad i=1,\ldots N,$$
which is real analytic on $\RR^N$ and satisfies $Y(0)=1$. Another definition of the Dunkl exponential can be given in terms of the intertwining operator $V_k$ which connects Dunkl operators to usual derivatives via the relation
$$ T_i V_k = V_k \partial_i.$$
The Dunkl exponential can then be equivalently defined as
$$ E_k(x,y) = V_k \left( e^{\langle \cdot, y \rangle} \right) (x).$$

The following growth estimates on $E_k$ are known: for all $x\in\RR^N$, $y\in\CC^N$ and all $\beta\in\ZZ^N_+$ we have
$$ |\partial_y^\beta E_k(x,y)| \leq |x|^{|\beta|} \displaystyle\max_{g\in G} e^{\text{Re}\langle gx,y\rangle}.$$

It is then possible to define a Dunkl transform on $L^1(\mu_k)$ by
$$ \mathcal{D}_k(f)(\xi)= \frac{1}{M_k}\Intr_{\RR^N} f(x)E_k(-i\xi,x)\diff \mu_k(x), \qquad \text{ for all } \xi\in\RR^N,$$
where $M_k$ is the Macdonald-Mehta integral. The Dunkl transform extends to an isometric isomorphism of $L^2(\mu_k)$; in particular, the Plancherel formula holds. When $k=0$ the Dunkl transform reduces to the Fourier transform.

The Dunkl heat kernel is defined as 
$$ h_t(x,y)=\frac{1}{M_k(2t)^{\gamma+N/2}}e^{-(|x|^2+|y|^2)/4t}E_k\left(\frac{x}{\sqrt{2t}},\frac{y}{\sqrt{2t}}\right),$$
for $t>0$ and $x,y\in\RR^N$, and it satisfies the bounds
\begin{align} \label{heatkernelbounds}
0< h_t(x,y)
\leq \frac{1}{(2t)^{\gamma+N/2}M_k}\displaystyle\max_{g\in G} e^{-|gx-y|^2/4t},
\end{align}
for all $x,y\in\RR^N$. 

The heat semigroup is defined for uniformly continuous $f:\RR^N\to \RR$ and $t\geq 0$ by
$$ e^{t\Delta_k}f(x)= \begin{cases} 
	\Intr_{\RR^N} h_t(x,y) f(y) w_k(y) \diff y, &\text{if } t>0, \\
	f(x), &\text{if } t=0.
	\end{cases} $$

\section{The carr\'e-du-champ operator}

A key ingredient in our proofs below will be the carr\'e-du-champ operator. This is defined as
$$\Gamma (f) = \frac{1}{2} (\Delta_k(f^2) - 2 f \Delta_k f).$$

The following Lemma gives an expression of this operator.

%%%%%%%%%%%%%%%%%%%%%%%%%%%%%%%%%%%%%%%%%%%%%%%%%%%%%%%%%%%%%%%%%%
\begin{lem} \label{Dunklcarreduchamp}
We have
$$ \Delta_k (f^2) 
 = 2f\Delta_k f + 2 |\nabla f|^2 + 2\Suma k_\alpha 
 	\left(
 		\frac{f(x)- f(\sigma_\alpha x)}{\alx}
 	\right)^2.$$
In particular, we obtain the following expression for the carr\'e-du-champ operator
\begin{equation*} 
\Gamma(f) = |\nabla f|^2 + \Suma k_\alpha 
 	\left(
 		\frac{f(x)- f(\sigma_\alpha x)}{\alx}
 	\right)^2.
\end{equation*}
\end{lem}
%%%%%%%%%%%%%%%%%%%%%%%%%%%%%%%%%%%%%%%%%%%%%%%%%%%%%%%%%%%%%%%%%%

\begin{proof}
Using (\ref{Dunkllaplacian}), we have
\begin{align*}
\Delta_k (f^2) 
& = \Delta (f^2) + 2 \Suma k_\alpha 
	\left( 
		\frac{\langle \nabla(f^2), \alpha \rangle}{\alx} - \frac{f^2(x)-f^2(\sigma_\alpha x)}{\alx^2}	
	\right)
\\
& = 2f \Delta f + 2 |\nabla f|^2 + 4f(x) \Suma k_\alpha 
	\left(
		\frac{\langle\nabla f, \alpha \rangle}{\alx} 
		- \frac{f(x)-f(\sigma_\alpha x)}{\alx^2}
	\right)
\\
&\qquad
 + 2\Suma k_\alpha 
 	\left(
 		\frac{f(x)- f(\sigma_\alpha x)}{\alx}
 	\right)^2
\\
& = 2f\Delta_k f + 2 |\nabla f|^2 + 2\Suma k_\alpha 
 	\left(
 		\frac{f(x)- f(\sigma_\alpha x)}{\alx}
 	\right)^2.
\end{align*}
The expression for $\Gamma(f)$ then follows immediately from this and the definition.
\end{proof}

In the Euclidean case, as well as on Riemannian manifolds with Laplace-Beltrami operator, we have $\Gamma(f)=|\nabla f|^2$. The same is not true in the Dunkl case; however, using integration by parts, we can compute the Dirichlet form to obtain
$$ \IntN \Gamma(f) \diff\mu_k = \IntN |\nabla_k f|^2 \diff\mu_k.$$
This, together with Lemma \ref{Dunklcarreduchamp}, give the following useful relation.

%%%%%%%%%%%%%%%%%%%%%%%%%%%%%%%%%%%%%%%%%%%%%%%%%%%%%%%%%%%%%%%%%
\begin{lem} \label{dunkl-euclidean-dirichletform}
For all $f\in C^1_0(\RR^N)$ we have
$$ \IntN |\nabla_k f|^2 \diff\mu_k \geq \IntN |\nabla f|^2 \diff\mu_k.$$
\end{lem}
%%%%%%%%%%%%%%%%%%%%%%%%%%%%%%%%%%%%%%%%%%%%%%%%%%%%%%%%%%%%%%%%%

A similar very useful result that was already used in the proof of Sobolev inequality is the following.

%%%%%%%%%%%%%%%%%%%%%%%%%%%%%%%%%%%%%%%%%%%%%%%%%%%%%%%%%%%%%%%%%
\begin{lem} \label{modfunction}
For all $f\in C_0^1(\RR^N)$ we have
$$ \IntN |\nabla_k |f||^2 \diff\mu_k \leq \IntN |\nabla_k f|^2 \diff\mu_k.$$
\end{lem}
%%%%%%%%%%%%%%%%%%%%%%%%%%%%%%%%%%%%%%%%%%%%%%%%%%%%%%%%%%%%%%%%%

\begin{proof}
We note that $(|f(x)|-|f(\sigma_\alpha x)|)^2 \leq (f(x)-f(\sigma_\alpha x))^2$ holds for all $x\in\RR^N$ and all $\alpha\in R$. Since also $|\nabla |f|| = |\nabla f|$, then from Lemma \ref{Dunklcarreduchamp} we have
$$ \Gamma(|f|) \leq \Gamma(f).$$
Using Dirichlet forms as above, the conclusion follows.
\end{proof}

We also have the following pointwise estimate of the carr\'e-du-champ operator in terms of the Dunkl gradient.

\begin{prop} \label{Gammaboundgradient}
There exists a constant $C>0$ such that the following inequality holds for all functions $f$ for which all terms below are well-defined
$$ \Gamma(f) \geq C |\nabla_k f|^2.$$
\end{prop}

\begin{proof}
For $i=1,\ldots,N$, let 
$$ D_i f(x)= \Suma \alpha_i k_\alpha \frac{f(x)-f(\sigma_\alpha x)}{\alx},$$
so 
$$ T_if(x) = \partial_i f(x) + D_if(x).$$

From Lemma \ref{Dunklcarreduchamp}, we have (recall that $|\alpha|^2=2$)
$$ \Gamma(f)(x)
=\sum_{i=1}^N \left(|\partial_i f(x)|^2 + \frac{1}{2}\Suma \alpha_i^2 k_\alpha \left(\frac{f(x)-f(\sigma_\alpha x)}{\alx} \right)^2\right).$$
We will estimate the $\alpha$-sum first. Take $\tilde{C}:=\min_{\alpha\in R_+} \frac{1}{2k_\alpha}$. Here we use the convention that if $k_\alpha=0$, then $\frac{1}{k_\alpha}=\infty$. Since $k$ does not vanish identically and $R_+$ is a finite set, then $\tilde{C}$ is well-defined and finite. We then have
\begin{align*}
\frac{1}{2}\Suma \alpha_i^2 k_\alpha \left(\frac{f(x)-f(\sigma_\alpha x)}{\alx} \right)^2
&\geq \tilde{C} \Suma \alpha_i^2 k_\alpha^2 \left(\frac{f(x)-f(\sigma_\alpha x)}{\alx} \right)^2
\\
&\geq \frac{\tilde{C}}{\sqrt{|R_+|}} (D_if(x))^2.
\end{align*}

Going back to the carr\'e-du-champ operator, we now have
\begin{equation} \label{Gammaineq}
\Gamma(f)(x) 
\geq \sum_{i=1}^N \left(|\partial_i f(x)|^2 + \frac{\tilde{C}}{\sqrt{|R_+|}} (D_if(x))^2\right).
\end{equation}

The following inequality holds for all $x,y\in \RR$ and all $c>0$
\begin{equation} \label{elementaryineq}
x^2 + cy^2 \geq \frac{c}{c+1} (x+y)^2.
\end{equation}
Indeed, by rearranging the terms, this is equivalent to
$$ (x-cy)^2 \geq 0.$$

Using this inequality separately for each $i=1,\ldots, N$, from inequality (\ref{Gammaineq}) above we obtain that
$$ \Gamma(f)(x) 
\geq C \sum_{i=1}^N \left(\partial_i f(x) + D_if(x)\right)^2 = C |\nabla_k f(x)|^2,$$
where $C=\frac{\tilde{C}/\sqrt{|R_+|}}{1+\tilde{C}/\sqrt{|R_+|}}$ is the constant obtained from inequality (\ref{elementaryineq}).
\end{proof}
\section{Nash's inequality}

In this section we will prove the Nash inequality using an elementary method that exploits the Dunkl transform and its rich theory. Using a nice method of \cite{BCLS} we can prove that this in turn implies the Sobolev inequality.

%Before starting the proof of Nash inequality, we need some preliminary results on the Dunkl transform. In \cite{BSKO}, the authors introduce an abstract generalised form of the Fourier transform, which includes in a special case the Dunkl transform. Within this framework, some important inequalities from the classical Fourier transform theory were generalised in \cite{J}. From this latter paper we will use a Hausdorff-Young inequality which we give below in the particular case of the Dunkl transform.

%\begin{prop}[Hausdorff-Young inequality]
%Let $p\in (1,2]$ and $q>0$ such that $\frac{1}{p}+\frac{1}{q}=1$. %Then we have
%
%$$ \norm{\DT(f)}_{L^q(\mu_k)}\leq \norm{f}_{L^p(\mu_k)},$$
%
%for any $f\in L^p(\mu_k)$. 
%\end{prop}

%Using this result, we can now state and prove the Nash inequality.

\begin{prop}[Nash inequality]

There exists a constant $C>0$ that depends on $N$ and $k$ such that the following inequality holds for all $f\in L^1(\mu_k)$ such that $\nabla_k f \in L^2(\mu_k)$:
$$ \norm{f}_{L^2(\mu_k)}^{1+2/(2\gamma+N)}  
\leq C \norm{\nabla_k(f)}_{L^2(\mu_k)} \norm{f}_{L^1(\mu_k)}^{2/(2\gamma+N)}.$$
\end{prop}

\begin{proof}
Fix an $R>0$ and let $B_R$ be the ball in $\RR^N$ of radius $R$ and centred at the origin. We have
\begin{align*}
\Intr_{\RR^N\setminus B_R} |\DT(f)(\xi)|^2 \diff\mu_k 
&\leq \Intr_{\RR^N\setminus B_R}\frac{|\xi|^2}{R^2} |\DT(f)(\xi)|^2 \diff\mu_k \\
&\leq \frac{1}{R^2} \Intr_{\RR^N} |\xi|^2 |\DT(f)(\xi)|^2 \diff\mu_k.
\end{align*}

Using the property that 
$$ \DT(T_jf)(\xi)= i\xi_j \DT(f)(\xi)$$
for any $j=1,\ldots, N$, then the right hand side of the above double inequality becomes
$$ \frac{1}{R^2} \Intr_{\RR^N} |\DT(\nabla_kf)|^2 \diff\mu_k
=\frac{1}{R^2} \Intr_{\RR^N} |\nabla_k(f)|^2 \diff\mu_k,$$
where we used Parseval's theorem. 

On the other hand, looking at the integral on $B_R$, we have
\begin{align*}
\Intr_{B_R} |\DT(f)(\xi)|^2 \diff\mu_k 
&= \frac{1}{M_k^2}\Intr_{B_R} \left| \Intr_{\RR^N} f(x)E_k(-i\xi,x)\diff\mu_k(x)\right|^2 \diff\mu_k(\xi)  
\\
&\leq \frac{1}{M_k^2}\Intr_{B_R} \left( \Intr_{\RR^N} |f(x)E_k(-i\xi,x)| \diff\mu_k(x) \right)^2 \diff\mu_k(\xi) 
\\
&\leq \frac{1}{M_k^2}\Intr_{B_R}\left(\Intr |f(x)| \diff\mu_k(x)\right)^2 \diff\mu_k(\xi) 
\\
&=\frac{1}{M_k^2}\mu_k(B_R) \norm{f}_{L^1(\mu_k)}^2.
\end{align*}
Here we used the bounds $|E_k(-i\xi,x)|\leq 1$ of the Dunkl exponential. 

Using the homogeneity of the Dunkl weight, we can compute using spherical coordinates
$$ \mu_k(B_R)=\Intr_{B_R} \diff\mu_k(x)=\int_{\mathbb{S}_{N-1}}\Intr_0^R r^{N+2\gamma-1} w_k(\theta) \diff r\diff \sigma(\theta) =\frac{p(B_1)}{N+2\gamma} R^{N+2\gamma},$$
where $p(B_1)=\Intr_{\mathbb{S}_{N-1}} w_k(\theta) \diff\theta$ (we come back to this constant in section 6).

Putting the above together and using Parseval's theorem, we have obtained
\begin{align*}
\Intr_{\RR^N} |f(x)|^2\diff\mu_k(x) 
&=\Intr_{\RR^N} |\DT(f)(\xi)|^2 \diff\mu_k(\xi) \\
&\leq \frac{1}{R^2} \Intr_{\RR^N} |\nabla_k(f)|^2 \diff\mu_k + \frac{p(B_1)}{M_k^2(N+2\gamma)} R^{N+2\gamma} \norm{f}_{L^1(\mu_k)}^2
\end{align*}
The right hand side is optimised for 
$$ R=\left( \frac{2M_k^2 \norm{\nabla_k(f)}_{L^2(\mu_k)}^2}{p(B_1) \norm{f}_{L^1(\mu_k)}^2}\right)^{1/(N+2\gamma+2)}$$
and upon substituting this above and raising everything to the power $\frac{2\gamma+N+2}{2(2\gamma+N)}$, we obtain finally
$$ \norm{f}_{L^2(\mu_k)}^{1+2/(2\gamma+N)}  
\leq C \norm{\nabla_k(f)}_{L^2(\mu_k)} \norm{f}_{L^1(\mu_k)}^{2/(2\gamma+N)},$$
for a constant $C$ which can be computed explicitly from the above.
\end{proof}

We can now deduce a Sobolev inequality using the elementary argument from \cite{BCLS}.

\begin{thm}[Sobolev inequality] \label{sobolev}
Suppose $N+2\gamma>2$. Then there exists a constant $C>0$ such that for all $f\in C_c^\infty(\RR^N)$ we have the inequality
$$ \norm{f}_{L^q(\mu_k)} \leq C \norm{\nabla_k f}_{L^2(\mu_k)},$$
where $q=\frac{2(N+2\gamma)}{N+2\gamma-2}$.
\end{thm}

As announced above, we will prove that the Sobolev inequality follows from the Nash inequality. Before giving the proof, we note that the opposite implication also holds. Indeed, this follows from H\"older's inequality $\norm{FG}_{L^1(\mu_k)} \leq \norm{F}_{L^P(\mu_k)}\norm{G}_{L^Q(\mu_k)}$ by applying it to $F=f^{2\theta}$, $G=f^{2(1-\theta)}$, with a suitable choice of $\theta\in(0,1)$ and $P,Q$ conjugate. We obtain
$$ \Intr_{\RR^N} f^2\diff\mu_k \leq \left(\Intr_{\RR^N} f^{2P\theta}\diff\mu_k\right)^{1/P}\left(\Intr_{\RR^N} f^{2(1-\theta)Q}\diff\mu_k\right)^{1/Q}.$$
Since we need $L^1$ and $L^{2(N+2\gamma)/(N+2\gamma-2)}$ norms on the right hand side, we let 
$$ 2\theta P=1 \;\;\text{ and }\;\; 2(1-\theta)Q=\frac{2(N+2\gamma)}{N+2\gamma-2}.$$
These equations, together with $\frac{1}{P}+\frac{1}{Q}=1$, have a unique solution. Making the substitution and using the Sobolev inequality, the Nash inequality follows.

\begin{proof}
Fix first a smooth function $f\geq 0$ of compact support. Let 
$$ f_j(x)=\begin{cases} 0, &\text{for } f(x)< 2^j \\
f(x)-2^j, &\text{for } 2^j\leq f(x) \leq 2^{j+1} \\
2^j, & \text{for } f(x)>2^{j+1}
\end{cases}$$
for any $j\in\ZZ$. Then $f_j\in L^1(\RR^N)$ and $\nabla_k f\in L^2(\mu_k)$, so we can apply the Nash inequality obtained above to get
\begin{align} \label{nashsob}
\left(\Intr_{\RR^N} f_j^2\diff\mu_k\right)^{1+2/(2\gamma+N)} 
\leq C^2 \Intr_{S_j\setminus S_{j+1}} |\nabla_k f|^2 \diff\mu_k \cdot \left( \Intr_{\RR^N} f_j \diff\mu_k\right)^{4/(2\gamma+N)},
\end{align}
where $S_j=f^{-1}([2^j,\infty))$. Using the definition of $f_j$, we also see that 
\begin{align} \label{bound1}
\Intr_{\RR^N} f_j^2 \diff\mu_k \geq 2^{2j} \mu_k(S_{j+1}),
\end{align}
and 
\begin{align} \label{bound2}
\Intr_{\RR^N} f_j \diff\mu_k \leq 2^j \mu_k(S_j).
\end{align}
Let $q=\frac{2(N+2\gamma)}{N+2\gamma-2}$, so we want to estimate $\norm{f}_{L^q(\mu_k)}$. We have
\begin{align*}
\Intr_{\RR^N} & f^q \diff\mu_k 
= \Sumn_{j\in\ZZ} \Intr_{S_j\setminus S_{j+1}} f^q\diff\mu_k 
\\
&\leq \Sumn_{j\in\ZZ} 2^{q(j+1)} (\mu_k(S_j)-\mu_k(S_{j+1}))
= (2^q-1)\Sumn_{j\in\ZZ} 2^{qj} \mu_k(S_j). 
\end{align*}
This suggests introducing the notation $a_j=2^{qj}\mu_k(S_j)$ and $b_j=\Intr_{S_j\setminus S_{j+1}} |\nabla_k f|^2 \diff\mu_k$. In light of the above, now we need to estimate $\Sumn_{j\in\ZZ} a_j$.

Using the bounds (\ref{bound1}) and (\ref{bound2}) into the inequality (\ref{nashsob}), we obtain after some manipulation
$$ a_{j+1} \leq 2^qC^{2p} b_j^p a_j^{4/(N+2\gamma+2)},$$
where $p=\frac{N+2\gamma}{N+2\gamma+2}$. Summing up over all $j\in\ZZ$ and noting that $\frac{2}{N+2\gamma+2}=1-p$, we can apply H\"older's inequality to obtain
\begin{align*}
\Sumn_{j\in\ZZ} a_{j}=\Sumn_{j\in\ZZ} a_{j+1} 
&\leq 2^qC^{2p} \Sumn_{j\in\ZZ} b_j^p a_j^{4/(N+2\gamma+2)} \\
&\leq 2^qC^{2p} \left( \Sumn_{j\in\ZZ} b_j\right)^p \left(\Sumn_{j\in\ZZ} a_j^2 \right)^{1-p}\\
&\leq 2^qC^{2p} \left( \Sumn_{j\in\ZZ} b_j\right)^p \left(\Sumn_{j\in\ZZ} a_j \right)^{2(1-p)}.
\end{align*}
Thus, we have
\begin{align*}
\norm{f}_q^q
\leq (2^q-1)\Sumn_{j\in\ZZ} a_{j} 
\leq 2^{q/(2p-1)} (2^q-1)C^{2p/(2p-1)} \left( \Sumn_{j\in\ZZ} b_j\right)^{p/(2p-1)}.
\end{align*}
But clearly $\Sumn_{j\in\ZZ} b_j\leq \Intr_{\RR^N} |\nabla_k f|^2\diff\mu_k$, and since $\frac{p}{2p-1}=\frac{q}{2}$, we obtain
$$ \norm{f}_{L^q(\mu_k)} \leq 2^{1/(2p-1)} (2^q-1)^{1/q}C \norm{\nabla_k f}_{L^2(\mu_k)},$$
as required.

If $f$ is not non-negative, then applying the above result to the non-negative function $|f|$ and using Lemma \ref{modfunction}, we have
$$ \norm{f}_q \leq \tilde{C} \norm{\nabla_k|f|}_2 \leq \tilde{C} \norm{\nabla_k f}_2.$$
This completes the proof.
\end{proof}

\section{Pseudo-Poincar\'e inequality}

In this section we will prove a sharper Sobolev-type inequality, involving  Besov spaces. Besov spaces $B_{pq}^s$ are generalisations of the classical Sobolev spaces; for example, we have $B_{2,2}^s=H^s$ for any $s>0$. They can be characterised by many equivalent definitions which are all rather technical; see \cite{Tri} for more information. Here, we will be concerned only with the spaces $B_{\infty,\infty}^s$, for $s<0$, which admit a simpler definition in terms of the heat semigroup. 

From now on we will work only in the Dunkl setting and in order to simplify notation, we will also use $B_{\infty,\infty}^s$ to denote the corresponding Dunkl-Besov space. Recall that $P_t = e^{t\Delta_k}$ is the heat semigroup. For any $s<0$, we define the Dunkl-Besov space $B_{\infty,\infty}^s$ as the space of all tempered distributions  $f\in \mathcal{S}'(\RR^N)$ for which the norm
$$ \norm{f}_{B_{\infty,\infty}^s} := \sup_{t>0} t^{-s/2} \norm{P_tf}_\infty$$
is finite. The aim of this section is to prove the following improved Sobolev inequality.

\begin{thm} \label{besovineq}
Let $1\leq p<q<\infty$. For any $f$ such that $\norm{\nabla_k f}_p<\infty$ we have the inequality
$$ \norm{f}_q \leq C \norm{\nabla_k f}_p^{p/q} \norm{f}_{B_{\infty,\infty}^{p/(p-q)}}^{1-p/q},$$
where $C>0$ is a constant.
\end{thm}

This inequality, in the classical case of simple derivatives, was first proved in \cite{CDVPX} using wavelets. A simplified proof was given by Ledoux in \cite{L}. Similar inequalities were also studied in \cite{CDDDV} and \cite{BEHL}. Here we will follow Ledoux's method that makes use of the known bounds on the Dunkl heat kernel.

The essential ingredient in this proof is the following pseudo-Poincar\'e inequality.

%%%%%%%%%%%%%%%%%%%%%%%%%%%%%%%%%%%%%%%%%%%%%%%%%%%%%%%%%%%%%%%%%
\begin{prop} \label{pseudopoincare}
For any $1\leq p \leq 2$, there exists a constant $C>0$ such that for all $f$ we have
$$ \norm{f-P_tf}_p \leq C\sqrt{t} \norm{\nabla_k f}_p.$$
\end{prop}
%%%%%%%%%%%%%%%%%%%%%%%%%%%%%%%%%%%%%%%%%%%%%%%%%%%%%%%%%%%%%%%%%

This type of inequality is interesting in its own right as it has proven to be a useful tool for proving Sobolev inequalities in particularly difficult geometric settings. It can be stated more generally as 
$$ \norm{f-A_r f}_p \leq C r \norm{\nabla f}_p,$$
where $A_r$ is some sort of averaging operator. See \cite{SC} and \cite{SC2} (and references therein) for more details, as well as some historical remarks.

Before we start proving the pseudo-Poincar\'e inequality, we need the following two lemmas. The first of these concerns gradient bounds, while the second one provides a Poincar\'e-type inequality for the semigroup $P_t$. 

\begin{lem} \label{gradientbound}
The following inequality holds
$$ |\nabla_k P_tf| \leq C P_t(|\nabla_kf|) $$
for a constant $C>0$. 
\end{lem}

\begin{proof}
We first note that since the Dunkl operators commute, then $\nabla_k$ and $P_t$ also commute. Then, using the inequality $\sum_i a_i^2 \leq \left(\sum_i a_i\right)^2$ which holds for nonnegative real numbers $a_i$, and then Cauchy-Schwarz, we obtain
\begin{align*}
|\nabla_k P_tf| (x)
= |P_t \nabla_k f| (x)
&= \left( \sum_{i=1}^N \left(\int_{\RR^N} T_if(y) h_t(x,y)  \diff\mu_k(y) \right)^2\right)^{1/2} 
\\
&\leq \sum_{i=1}^N \int_{\RR^N}|T_if(y)| h_t(x,y) \diff\mu_k(y) 
\\
&\leq \sqrt{N} \int_{\RR^N} |\nabla_k f(y)| h_t(x,y)\diff\mu_k(y)
\\
&=\sqrt{N} P_t(|\nabla_k f|)(x),
\end{align*}
as required. 
\end{proof}

\begin{lem} \label{reversePoincare}
The following inequality holds
$$ P_t(f^2) - (P_tf)^2 \geq Ct |\nabla_k P_t f|^2,$$
for a constant $C>0$.
\end{lem}

\begin{proof}
We note that
\begin{align*}
P_t(f^2) - (P_tf)^2
&= \int_0^t \frac{\diff}{\diff s} (P_s((P_{t-s}f)^2)) \diff s
\\
&=\int_0^t P_s \left( \Delta_k((P_{t-s}f)^2)-2P_{t-s}f\Delta_k(P_{t-s}f) \right) \diff s
\\
&=2\int_0^t P_s\left( \Gamma(P_{t-s}f) \right) \diff s.
\end{align*}
Using first Proposition \ref{Gammaboundgradient}, and then Cauchy-Schwarz, this becomes
\begin{align*}
P_t(f^2) - (P_tf)^2
&\geq 2C\int_0^t P_s (|\nabla_k P_{t-s}f|^2) \diff s
\\
&\geq 2C\int_0^t (P_s (|\nabla_k P_{t-s} f|))^2 \diff s.
\end{align*}
Finally, using Lemma \ref{gradientbound}, we obtain
\begin{align*}
P_t(f^2) - (P_tf)^2
&\geq \frac{2C}{\sqrt{N}} \int_0^t |\nabla_k P_t f|^2 \diff s
\\
&= \frac{2C}{\sqrt{N}} t |\nabla_k P_t f|^2.
\end{align*}
\end{proof}

We can now prove the pseudo-Poincar\'e inequality. 

\begin{proof} [Proof of Proposition \ref{pseudopoincare}]
We note that 
$$ P_tf -f 
= \int_0^t \frac{\diff}{\diff s} (P_s f) \diff s
= \int_0^t \Delta_k P_sf \diff s.$$
Fix a smooth function $g$ such that $\norm{g}_{p^*}\leq 1$, where $\frac{1}{p}+\frac{1}{p^*}=1$. Then
\begin{align*}
\int_{\RR^N} g(f-P_tf)\diff\mu_k 
&=\int_{\RR^N} (g-P_tg)f\diff\mu_k 
\\
&=-\int_0^t \int_{\RR^N} f \Delta_kP_sg \diff\mu_k \diff s
\\
&= \int_0^t \int_{\RR^N} \nabla_k f \cdot \nabla_kP_sg \diff\mu_k \diff s
\\
&\leq \norm{\nabla_k f}_p \int_0^t \norm{\nabla_k P_sg}_{p^*} \diff s.
\end{align*}
From Lemma \ref{reversePoincare} we have 
$$ P_s(g^2) \geq P_s(g^2)-(P_sg)^2 \geq Ct|\nabla_kP_s g|^2.$$
Since $p^*\geq 2$, then the function $|\cdot|^{p^*/2}$ is convex and so, using Jensen's inequality, 
$$ \int_{\RR^N} |\nabla_k P_sg|^{p^*} \diff\mu_k 
\leq (Cs)^{-p^*/2} \int_{\RR^N} P_s(g^2)^{p^*/2} \diff\mu_k 
\leq (Cs)^{-p^*/2} \int_{\RR^N} P_s(|g|^{p^*}) \diff\mu_k,$$
and thus
$$ \norm{\nabla_kP_s g}_{p^*} 
\leq \frac{1}{\sqrt{Cs}} \norm{P_s(|g|^{p^*})}_1^{1/p^*} 
\leq \frac{1}{\sqrt{Cs}} \norm{g}_{p^*}
\leq \frac{1}{\sqrt{Cs}}.$$
Here we used the fact that $P_t$ is a contraction on $L^1(\mu_k)$. Therefore
$$ \int_{\RR^N} g(f-P_tf)\diff\mu_k  \leq C \sqrt{t} \norm{\nabla_k f}_p.$$
Since $g$ was arbitrary, this shows that 
$$ \norm{f-P_tf}_p \leq C\sqrt{t} \norm{\nabla_k f}_p,$$
as required.
\end{proof}

We are now in a position to give a proof of the improved Sobolev inequality.

\begin{proof} [Proof of Theorem \ref{besovineq}]
The proof goes in three steps: the first step is proving a weak form of the inequality, the second step is proving the inequality for $f$ that satisfies the additional assumption $f\in L^q$, and finally, in the last step, we remove the assumption $f\in L^q$ from the previous step.

\textbf{Step 1.} In this step we establish the weak inequality
\begin{align} \label{weakineq}
\norm{f}_{q,w} \leq C \norm{\nabla_k f}_p^\theta \norm{f}^{1-\theta}_{B_{\infty,\infty}^{\theta/(\theta-1)}},
\end{align}
where $\theta = \frac{p}{q}\in (0,1)$, and the weak $L^q$ norm is defined by
$$ \norm{f}_{q,w}^q:=\sup_{t>0} t^q \mu_k\left( \{ |f(x)|>t\}\right).$$
By homogeneity, we can assume that $\norm{f}_{B_{\infty,\infty}^{\theta/(\theta-1)}}\leq 1$, i.e., 
\begin{equation} \label{ptbound}
|P_tf|\leq t^{\theta/2(\theta-1)} \qquad \text{ for all } t>0.
\end{equation}
For every $t>0$, let $s_t=t^{2(\theta-1)/\theta}$, so, by (\ref{ptbound}), we have $|P_{s_t}f|\leq t$.

If $|f|\geq 2t$, then $|f-P_{s_t}f|\geq |f| - |P_{s_t}f| \geq t$, so
\begin{equation}
\mu_k \left( \{ |f(x)| \geq 2t \} \right) 
\leq \mu_k \left( |f-P_{s_t}f| \geq t \right) 
\leq t^{-p} \int_{\RR^N} |f-P_{s_t}f|^p \diff\mu_k.
\end{equation}
Using the pseudo-Poincar\'e inequality from Proposition \ref{pseudopoincare}, this implies that
\begin{equation}
t^q \mu_k \left( \{ |f(x)| \geq 2t \} \right)
\leq C t^{q-p} s_t^{p/2} \norm{\nabla_k f}_p^p.
\end{equation}
But from the choice of $s_t$, and because $\theta=\frac{p}{q}$, we have $t^{q-p}s_t^{p/2}=1$, so taking infimum over $t>0$, by the definition of the weak $L^q$ norm, we finally obtain (\ref{weakineq}).

\textbf{Step 2.} Here we will impose the additional assumption $f\in L^q$. As before, we can assume by homogeneity that $\norm{f}_{B_{\infty,\infty}^{\theta/(\theta-1)}}= 1$, so we need to prove that 
\begin{equation} \label{besovstep2} 
\int_{\RR^N} |f|^q \diff\mu_k \leq C \int_{\RR^N} |\nabla_k f|^p \diff\mu_k.
\end{equation}
Keeping with the notation above, for any $t>0$ we let $s_t=t^{2(\theta-1)/\theta}$ and so $|P_{s_t}f| \leq t$.

Using the cake layer representation theorem (see \cite[Theorem 1.13]{LL}) we have
\begin{equation} \label{cakelayer5}
\int_{\RR^N} |f|^q \diff\mu_k = 5^q \int_0^\infty \mu_k \Big( |f| \geq 5t \Big) \diff (t^q).
\end{equation}

Fix a constant $c\geq 4$. Define the function 
$$ \tilde{f}_t = (f-t)_+ \wedge (ct) + (f+t)_- \vee (-ct).$$
If $|f(x)|\geq 5t$, then either $|\tilde{f}(x)|=|f(x)|-t$, or $|\tilde{f}|=ct$. It then follows that (recall that we required $c\geq 4$)
$$ \{ |f|\geq 5t \} \subset \{|\tilde{f}|\geq 4t\},$$
and thus 
\begin{align} \label{5t-4t}
\int_0^\infty \mu_k\Big( |f|\geq 5t\Big) \diff (t^q) 
\leq \int_0^\infty \mu_k\left( |\tilde{f}|\geq 4t\right) \diff (t^q).
\end{align}
The triangle inequality implies that
\begin{align*}
|\tilde{f}| 
&\leq |\tilde{f}-P_{s_t}\tilde{f}| + |P_{s_t}(\tilde{f}-f)| + |P_{s_t}f|
\\
&\leq |\tilde{f}-P_{s_t}\tilde{f}| + P_{s_t}(|\tilde{f}-f|) + t,
\end{align*}
where we used the convexity of the modulus function and the fact that $|P_{s_t}f|\leq t$. It follows that
$$ \{|\tilde{f}| \geq 4t \} \subset \{ |\tilde{f}-P_{s_t}\tilde{f}| \geq t\} \cup \{P_{s_t}(|\tilde{f}-f|) \geq 2t\}.$$
Hence
\begin{equation} \label{4tineq}
\begin{aligned}
\int_0^\infty \mu_k\left( |\tilde{f}|\geq 4t\right) \diff (t^q)
&\leq 
	\int_0^\infty \mu_k \left( |\tilde{f}-P_{s_t}\tilde{f}| \geq t\right) \diff (t^q)
\\
&\qquad +
	\int_0^\infty \mu_k \left( P_{s_t}(|\tilde{f}-f|) \geq 2t\right) \diff (t^q).
\end{aligned}
\end{equation}
We will estimate the two terms appearing on the right hand side of this inequality separately. Firstly, 
\begin{align*}
\mu_k \left( |\tilde{f}-P_{s_t}\tilde{f}| \geq t\right)
&\leq \int_{\RR^N} \frac{|\tilde{f}-P_{s_t}\tilde{f}|^p}{t^p} \diff\mu_k 
\\
&\leq C t^{-p} s_t^{p/2} \int_{\RR^N} |\nabla_k \tilde{f}|^p \diff\mu_k,
\end{align*}
where in the second line we used the pseudo-Poincar\'e inequality from Proposition \ref{pseudopoincare}. We note that by the construction of $\tilde{f}$ we have that $\nabla_k \tilde{f} = \nabla_k f$ on $ t \leq |f| \leq (c+1)t$, and $\nabla_k \tilde{f}=0$ otherwise. Consequently, 
\begin{align*}
\int_0^\infty \mu_k \left( |\tilde{f}-P_{s_t}\tilde{f}| \geq t\right) \diff (t^q)
&\leq C \int_0^\infty t^{-q} \int_{\RR^N} \1_{\{t\leq |f| \leq (c+1)t\}} |\nabla_k f|^p \diff\mu_k \diff (t^q)
\\
&= Cq \int_{\RR^N} |\nabla_k f|^p \int_{|f|/(c+1)}^{|f|} \frac{1}{t} \diff t \diff\mu_k
\\
&=Cq \log (c+1) \int_{\RR^N} |\nabla_k f|^p \diff\mu_k.
\end{align*}

Secondly, in order to estimate the second term appearing in (\ref{4tineq}) we first note that
\begin{align*}
|f-\tilde{f}| 
&= |f-\tilde{f}| \1_{\{|f|\leq (c+1)t\}} + |f-\tilde{f}| \1_{\{|f| \geq (c+1)t \}}
\\
&\leq t + |f| \1_{\{|f| \geq (c+1)t \}},
\end{align*}
and using this we have
\begin{align*}
\int_0^\infty \mu_k \left( P_{s_t}(|\tilde{f}-f|) \geq 2t\right) \diff (t^q)
&\leq \int_0^\infty \mu_k \Big( P_{s_t}(|f|\1_{\{|f| \geq (c+1)t \}}) \geq t\Big) \diff (t^q)
\\
&\leq \int_0^\infty \int_{\RR^N} \frac{|f|\1_{\{|f| \geq (c+1)t \}}}{t} \diff\mu_k \diff (t^q)
\\
&=\int_{\RR^N} |f| \int_0^{|f|/(c+1)} q t^{q-2} \diff t \diff\mu_k
\\
&=\frac{q}{q-1} \frac{1}{(c+1)^{q-1}} \int_{\RR^N} |f|^q \diff\mu_k.
\end{align*}

Putting these back in (\ref{4tineq}), and also using (\ref{5t-4t}) and (\ref{cakelayer5}), we finally have
\begin{align*}
\int_{\RR^N} |f|^q \diff\mu_k
\leq C q5^q \log (c+1) \int_{\RR^N} |\nabla_k f|^p \diff\mu_k
+ \frac{q}{q-1} \frac{5^q}{(c+1)^{q-1}} \int_{\RR^N} |f|^q \diff\mu_k.
\end{align*}
Choosing $c\geq 4$ large enough so that $\frac{q 5^q}{q-1} <(c+1)^{q-1}$, we obtain (\ref{besovstep2}) for a constant $C>0$. 

\textbf{Step 3.} This is a technical step in which we remove the unnecessary assumption $f\in L^q$ from the previous step. The proof also follows that of Step 2.

Let $f\in W^{1,p}(\RR^N)$ and, as before, we can assume by homogeneity that $\norm{f}_{B_{\infty,\infty}^{\theta/(\theta-1)}}= 1$. Fix also $0<\epsilon<1$ and define
$$ N_\epsilon(f) = \int_\epsilon^{1/\epsilon} \mu_k \Big( |f| \geq 5t\Big) \diff (t^q).$$
By the weak inequality (\ref{weakineq}) this is a finite quantity. Moreover, following the proof of Step 2, we can bound this as
\begin{align} \label{Nepsilon1}
N_\epsilon(f) \leq Cq \log(c+1) \int_{\RR^N} |\nabla_k f|^p \diff\mu_k 
+ \int_\epsilon^{1/\epsilon} \frac{1}{t} \int_{\RR^N} |f| \1_{\{|f|\geq (c+1)t\}} \diff\mu_k \diff(t^q).
\end{align}
We now estimate the second term on the right hand side of this inequality. Changing the order of integration using Fubini's theorem, we have
\begin{align*}
\int_\epsilon^{1/\epsilon} \frac{1}{t} & \int_{\RR^N} |f| \1_{\{|f|\geq (c+1)t\}} \diff\mu_k \diff(t^q) 
\\
&=q\int_{\RR^N} |f| 
	\left[ 
		\1_{\{\epsilon \leq \frac{|f|}{c+1} \leq 1/\epsilon\}} \int_\epsilon^{|f|/(c+1)} t^{q-2} \diff t
		+\1_{\{\frac{|f|}{c+1}>1/\epsilon\}} \int_\epsilon^{1/\epsilon} t^{q-2} \diff t
	\right]\diff\mu_k
\\
&=\frac{q}{q-1} \int_{\RR^N} 
	\left[
		\1_{\{\epsilon \leq \frac{|f|}{c+1} \leq 1/\epsilon\}} \left( \frac{|f|^{q}}{(c+1)^{q-1}} - \epsilon^{q-1}|f| \right) 
		+ \1_{\{\frac{|f|}{c+1}>1/\epsilon\}} \left( \frac{|f|}{\epsilon^{q-1}}-\epsilon^{q-1}|f|\right)
	\right] \diff\mu_k.
\end{align*}
Using a similar method to the above, we also note that 
\begin{align*}
\int_\epsilon^{1/\epsilon}  \mu_k &\Big( |f|\geq (c+1)t \Big) \diff(t^q) 
\\
&= \int_\epsilon^{1/\epsilon} \int_{\RR^N} \1_{\{|f|\geq (c+1)t\}} \diff\mu_k \diff(t^q)
\\
&=\int_{\RR^N} 
	\left[
		\1_{\{\epsilon \leq \frac{|f|}{c+1} \leq 1/\epsilon\}} \int_\epsilon^{|f|/(c+1)} \diff(t^q) 
		+ \1_{\{\frac{|f|}{c+1}>1/\epsilon\}} \int_\epsilon^{1/\epsilon} \diff(t^q)
	\right]\diff\mu_k
\\
&=\int_{\RR^N} 
	\left[ 
		\1_{\{\epsilon \leq \frac{|f|}{c+1} \leq 1/\epsilon\}} \left(\frac{|f|^q}{(c+1)^q}-\epsilon^q\right)
		+ \1_{\{\frac{|f|}{c+1}>1/\epsilon\}} \left( \frac{1}{\epsilon^q}-\epsilon^q\right)
	\right]\diff\mu_k.
\end{align*}

From these last two computations we can now deduce that 
\begin{align*}
\int_\epsilon^{1/\epsilon} &\frac{1}{t} \int_{\RR^N} |f| \1_{\{|f|\geq (c+1)t\}} \diff\mu_k \diff(t^q)
\leq \frac{q(c+1)}{q-1} \int_\epsilon^{1/\epsilon}  \mu_k \Big( |f|\geq (c+1)t \Big) \diff(t^q)
\\
&\qquad \qquad\qquad 
+ \frac{q}{q-1} \int_{\RR^N} \1_{\{\frac{|f|}{c+1}>1/\epsilon\}} \left( \frac{|f|}{\epsilon^{q-1}} -\frac{c+1}{\epsilon^q} + (c+1)\epsilon^q - |f| \epsilon^{q-1} \right) \diff\mu_k.
\end{align*}
We note also that
\begin{align*}
\int_{1/\epsilon}^\infty \mu_k \Big( |f|\geq (c+1)t \Big) \diff t 
&= \int_{1/\epsilon}^\infty \int_{\RR^N} \1_{\{|f|\geq (c+1)t\}}\diff\mu_k \diff t
\\
&=\int_{\RR^N} \1_{\{\frac{|f|}{c+1}>1/\epsilon\}} \left( \frac{|f|}{c+1}-\frac{1}{\epsilon}\right) \diff\mu_k.
\end{align*}
Finally, these give 
\begin{equation} \label{Nepsilon2}
\begin{aligned}
\int_\epsilon^{1/\epsilon} &\frac{1}{t} \int_{\RR^N} |f| \1_{\{|f|\geq (c+1)t\}} \diff\mu_k \diff(t^q)
\leq \frac{q(c+1)}{q-1} \int_\epsilon^{1/\epsilon}  \mu_k \Big( |f|\geq (c+1)t \Big) \diff(t^q)
\\
&\qquad \qquad\qquad 
+ \frac{q}{q-1} \frac{c+1}{\epsilon^{q-1}} \int_{1/\epsilon}^\infty \mu_k \Big( |f|\geq (c+1)t \Big) \diff t .
\end{aligned}
\end{equation}

At this point we use the definition of the weak norm $\norm{\cdot}_{q,w}$ and a change of variables to obtain
\begin{align*}
\int_\epsilon^{1/\epsilon}  \mu_k \Big( |f|\geq (c+1)t \Big) \diff(t^q)
\leq\left(\frac{5}{c+1}\right)^q N_\epsilon(f) + \frac{q}{(c+1)^q} \log \frac{c+1}{5} \norm{f}_{q,w}^q,
\end{align*}
and also
\begin{align*}
\int_{1/\epsilon}^\infty \mu_k \Big( |f|\geq (c+1)t \Big) \diff t
\leq \frac{\epsilon^{q-1}}{(c+1)^q (q-1)} \norm{f}_{q,w}^q.
\end{align*}

Using these in (\ref{Nepsilon2}) and then (\ref{Nepsilon1}), we have that
\begin{align*}
N_\epsilon(f) 
&\leq Cq\log(c+1) \norm{\nabla_k f}_p^p 
+ \frac{q}{q-1} \frac{5^q}{(c+1)^{q-1}} N_\epsilon(f) 
\\
&\qquad + \frac{q}{q-1} \frac{1}{(c+1)^{q-1}}  \left( q\log\frac{c+1}{5} + \frac{1}{q-1} \right) \norm{f}_{q,w}^q.
\end{align*}
From Step 1 we know that $\norm{f}_{q,w}<\infty$, and from our assumption we have $\norm{\nabla_k f}_p<\infty$. Thus, by choosing $c$ large enough (independent of $\epsilon$), we have $N_\epsilon(f)<\infty$. Therefore, taking the limit $\epsilon\to 0$, this implies that $\norm{f}_q<\infty$. We have now reduced the problem to Step 2, so the proof is complete. 
\end{proof}

We can show that the inequality of Theorem \ref{besovineq} implies the classical Sobolev and Gagliardo-Nirenberg inequalities, but first we need the following ultracontractivity result.

\begin{prop} \label{ultracontractivity}
Let $1\leq p\leq \infty$. Then, there exists a constant $C>0$ such that for any $f\in L^p(\mu_k)$ and for any $t>0$, we have
$$ \norm{P_tf}_\infty \leq C t^{-(2\gamma+N)/2p}\norm{f}_p.$$
\end{prop}

\begin{proof}
Firstly, if $f\in L^1(\mu_k)$, then, using the bounds on the Dunkl heat kernel from (\ref{heatkernelbounds}), we deduce that 
$$ |P_tf(x)| \leq \int_{\RR^N} \frac{1}{(2t)^{\gamma+N/2}M_k}|f(y)| w_k(y) \diff y = C t^{-\gamma-N/2} \norm{f}_1,$$
for a constant $C>0$ that does not depend on $t$. Thus the Proposition holds true for $p=1$.

On the other hand, if $f\in L^\infty(\mu_k)$, then, using the fact that $\int_{\RR^N} h_t(x,y) w_k(y) \diff y=1$ for any $x\in\RR^N$, we have
$$ |P_tf(x)| \leq \norm{f}_\infty,$$
so the Proposition holds true in the case $p=\infty$.

The general case then follows by interpolation, using the Riesz-Thorin Theorem (see for example Theorem 1.1.5 in \cite{Davies}).  
\end{proof}

Using this result we are now ready to present the classic Sobolev and Gagliardo-Nirenberg inequalities.

\begin{cor} [Sobolev inequality]
Let $1\leq p < \max\{2,N+2\gamma\}$ and define $q=\frac{p(N+2\gamma)}{N+2\gamma-p}$. Then, for any $f\in C_c^\infty(\RR^N)$, we have the inequality
$$ \norm{f}_q \leq C\norm{\nabla_k f}_p,$$
where $C>0$ is a constant.
\end{cor}

\begin{proof}
Let $\theta = \frac{p}{q} = \frac{N+2\gamma-p}{N+2\gamma}$. By Proposition \ref{ultracontractivity} we have
$$ \norm{P_tf}_\infty \leq C t^{-(2\gamma+N)/2q} \norm{f}_q,$$
so, by the definition of the Besov norm, we deduce that
$$ \norm{f}_{B_{\infty,\infty}^{-(2\gamma+N)/q}} \leq \norm{f}_q.$$
We note that $\frac{2\gamma+N}{q}=-\frac{N+2\gamma-p}{p}=-\frac{\theta}{\theta-1}$ so Theorem \ref{besovineq} gives
$$ \norm{f}_q \leq C \norm{\nabla_k f}_p^\theta \norm{f}_q^{1-\theta},$$
from which the Sobolev inequality follows immediately.
\end{proof}

\begin{cor} [Gagliardo-Nirenberg inequality]
Let $1\leq p <q <\infty$ such that also $p\leq 2$, and define $\frac{r}{q(N+2\gamma)} = \frac{1}{p} - \frac{1}{q}$. Then, for any $f\in C_c^\infty(\RR^N)$, we have the inequality
$$ \norm{f}_q \leq C\norm{\nabla_k f}_p^{p/q} \norm{f}_r^{1-p/q}.$$
\end{cor}

\begin{proof}
This follows from Theorem \ref{besovineq} as above since $\theta=\frac{p}{q}$ and $\frac{\theta}{\theta-1} = -\frac{N+2\gamma}{r}$.
\end{proof}

\section{Estimates of the best constant}

In this section we provide estimates for the best constant in the Dunkl Sobolev inequality, i.e., we study 
$$ C_{DS} := \sup \frac{\norm{f}_q}{\norm{\nabla_k f}_2},$$
where the supremum is considered over all non-zero functions in $C_c^1(\RR^N)$. From Lemma \ref{dunkl-euclidean-dirichletform} we see that 
$$ \frac{\norm{f}_q}{\norm{\nabla_k f}_2} \leq \frac{\norm{f}_q}{\norm{\nabla f}_2},$$
so we can obtain estimates by comparison to the optimal constant in the Sobolev inequality in the weighted space $L^2(\mu_k)$ but with usual gradient $\nabla$. More precisely, consider the inequality
$$ \norm{f}_q \leq C_{CS} \norm{\nabla_k f}_2,$$
with optimal constant, i.e.,
$$ C_{CS} := \sup \frac{\norm{f}_q}{\norm{\nabla f}_2}.$$
From the observation above we see straight away that
$$ C_{DS} \leq C_{CS}.$$

In what follows we will find the precise value of $C_{CS}$ by first proving an isoperimetric inequality for the weighted measure on Weyl chambers, and then obtaining a rearrangement inequality using a general result of Talenti \cite{T2} for weighted measures. This will provide the upper estimate of $C_{DS}$. For the lower estimate, we compute the supremum in the definition of $C_{DS}$ but over all radial functions. We end the section with a conjecture, that the supremum is actually achieved for a radial function.

The main result is the following:

\begin{thm} \label{sobolevconstantsthm}
With the same notation as above, we have
$$ |G|^{-\frac{1}{N+2\gamma}} C_{CS} \leq C_{DS} \leq C_{CS}.$$
More precisely, we have
\begin{align*}
&\left(\frac{2}{(N+2\gamma)(N+2\gamma-2)} \right)^{1/2} 
\cdot \left[ \frac{1}{M_k}\frac{\Gamma(N+2\gamma)}{\Gamma((N+2\gamma)/2)}\right]^{\frac{1}{N+2\gamma}} 
\\
&\qquad \qquad \qquad
\leq C_{DS}
\leq \left(\frac{2}{(N+2\gamma)(N+2\gamma-2)} \right)^{1/2} 
\cdot \left[ \frac{|G|}{M_k}\frac{\Gamma(N+2\gamma)}{\Gamma((N+2\gamma)/2)}\right]^{\frac{1}{N+2\gamma}}.
\end{align*}
\end{thm}

\subsection{Isoperimetric inequality}

For a measurable set $A\subset \RR^N$ we define
$$ p(A) := \int_{\partial A} w_k(x) \diff\sigma(x),$$
where $\sigma$ is the surface measure on $\partial A$. The aim of this subsection is to prove the following isoperimetric inequality.

\begin{thm}[Isoperimetric inequality] \label{iso}
Let $\Omega \subset \RR^N$ be a bounded Lipschitz domain. Then we have the inequality
$$ \mu_k(\Omega)^{1-\frac{1}{N+2\gamma}} \leq C p(\Omega),$$
for a constant $C=\frac{\mu_k(B_1^\epsilon)^{1-\frac{1}{N+2\gamma}}}{p(B_1^\epsilon)}$, where $\epsilon\in E$ is such that $\RR^N_\epsilon$ is a Weyl chamber, $B_1=\{|x|<1\}$ and $B_1^\epsilon=B_1\cap \RR^N_\epsilon$.
\end{thm}

\begin{proof}

First assume that $\Omega$ is contained in a Weyl chamber $\RR^N_\epsilon$ for some $\epsilon\in E$. Moreover, since $\Omega$ is a Lipschitz domain, it can be approximated by smooth domains that also approximate the perimeter and area (see, for example, \cite{D}), so we can assume further that $\overline{\Omega} \subset \RR^N_\epsilon$. We start by considering the Neumann problem 
\begin{equation} \label{neumannproblem}
\begin{cases}
w_k^{-1} \nabla \cdot (w_k \nabla u) = c & \text{ in } \Omega
\\
\frac{\partial u}{\partial \nu} = 1 & \text{ on } \partial\Omega,
\end{cases}
\end{equation}
where $\nu$ is outward normal to the boundary of $\Omega$, and $c$ is a constant. If the equation has a solution $u\in H^1(\Omega)$, then, by the divergence theorem, we must have
$$ \int_\Omega \nabla \cdot (w_k \nabla u) \diff x = \int_{\partial\Omega} w_k \frac{\partial u}{\partial \nu} \diff\sigma(x),$$
so $c=\frac{p(\Omega)}{\mu_k(\Omega)}$. Conversely, if $c=\frac{p(\Omega)}{\mu_k(\Omega)}$, then the general theory of elliptic equations tells us that the problem has a unique solution $u\in H^1(\Omega)$ (up to a constant). Moreover, by the assumption on $\Omega$ at the beginning of the proof, the operator
\begin{equation} \label{uniformllyellipticoperator}
 w_k^{-1} \nabla \cdot (w_k \nabla u) 
= \Delta u + \Suma \frac{2k_\alpha}{\alx} \langle \alpha, \nabla u\rangle 
\end{equation}
is uniformly elliptic on $\Omega$. Thus, by the regularity theory, the solution $u$ is smooth. 

Consider the set 
$$ \Gamma_u:=\{ x\in\Omega : u(y)-u(x) \geq \nabla u(x) \cdot (y-x) \text{ for all } y\in \overline{\Omega} \}.$$
Geometrically, this is the set of all points $x$ for which the graph of $u$ lies entirely above the tangent hyperplane at $x$. Consider also the set
$$ \Gamma^\epsilon_u := \left\{ x\in\Gamma_u : \text{sgn}\left(\langle \alpha, \nabla u(x) \rangle\right) = \epsilon(\alpha) \text{ for all } \alpha\in R_+ \right\}.$$
We will show that $B_1^\epsilon \subset \nabla u(\Gamma_u^\epsilon)$. Indeed, take $\zeta\in B_1^\epsilon$, and consider the point $x\in \overline{\Omega}$ for which
$$ \min_{y\in\overline{\Omega}} \left( u(y) - \zeta \cdot y \right)=u(x)-\zeta \cdot x.$$
If $x\in\partial\Omega$, then 
$$ \frac{\partial u(x)}{\partial \nu} 
= \lim_{h\downarrow 0} \frac{u(x-h\nu)-u(x)}{-h} \leq \zeta \cdot \nu \leq |\zeta| <1,$$
which contradicts the boundary condition. So $x\in \Omega$ and so, by definition, $x\in\Gamma_u$. Also, being a minimum of the function $u(y)-\zeta\cdot y$, by taking gradient we have
$$ \nabla u(x)=\zeta.$$
Since $\zeta \in B_1^\epsilon$, this shows that $x\in \Gamma_u^\epsilon$, and consequently $B_1^\epsilon \subset \nabla u(\Gamma_u^\epsilon)$. 

This implies that
$$ \mu_k(B_1^\epsilon) 
\leq \int_{\nabla u(\Gamma_u^\epsilon)} w_k(x) \diff x.$$
Using the change of variables $y=(\nabla u)^{-1} (x)$ (notice that by definition $\nabla u$ is injective on $\Gamma_u$), we have
$$ \int_{\nabla u(\Gamma_u^\epsilon)} w_k(x) \diff x
= \int_{\Gamma_u^\epsilon} w_k(\nabla u(y)) |\text{det}(H(u)(y)| \diff y.$$
Here $H(u)$ is the Hessian matrix of $u$, i.e., 
$$ H(u) = \left(\frac{\partial^2 u}{\partial x_i \partial x_j}\right)_{1\leq i,j \leq N}.$$
Denote by $\lambda_1(y), \ldots, \lambda_N(x)$ the eigenvalues of $H(u)(y)$. Then the Jacobian factor can also be written as
$$ |\text{det}(H(u)(y))| = \lambda_1(y) \cdots \lambda_N(y).$$

Before moving to the next step, we note that the Hessian matrix is positive-semidefinite on $\Gamma_u$. Indeed, assume that for some $x\in \Gamma_u$ there exists a vector $v \in \RR^N$ such that
$$ v^T \cdot H(u)(x) \cdot v <0.$$
By continuity of $H(u)$, there exists $r>0$ such that the same holds over a small ball $B_{2r}(x) \subset \Omega$, i.e.,
\begin{equation} \label{possemidef}
v^T \cdot H(u)(y) \cdot v <0 \qquad \text{ for all } y\in B_{2r}(x).
\end{equation}
By the mean value theorem, there exists $\theta \in (0,1)$ such that
$$ u(x+rv)-u(x) = r\nabla u(x) \cdot v + r^2 v^T \cdot H(u)(x+\theta r v) \cdot v.$$
By (\ref{possemidef}) we have
$$ u(x+rv)-u(x) < r \nabla u(x) \cdot v,$$
contradicting the fact that $x\in \Gamma_u$. Therefore $H(u)(x)$ is positive-semidefinite and so $\lambda_1(x), \ldots, \lambda_N(x) \geq 0$. 

Applying the weighted mean value inequality in the above, we have
\begin{align*}
\mu_k(B_1^\epsilon) 
& \leq \int_{\Gamma_u^\epsilon} w_k(\nabla u(y)) |\text{det}(H(u)(y)| \diff y
\\
& = \int_{\Gamma_u^\epsilon} \prod_{\alpha\in R_+} \left( \frac{\langle \alpha, \nabla u(y) \rangle}{\alx} \right)^{2k_\alpha} \cdot \lambda_1(y) \cdots \lambda_N (y) w_k(y) \diff y
\\
& \leq \int_{\Gamma_u^\epsilon} \frac{1}{(N+2\gamma)^{N+2\gamma}} \left[ \Suma 2k_\alpha \frac{\langle \alpha, \nabla u(y) \rangle}{\alx} + \lambda_1(y) + \cdots + \lambda_N(y) \right]^{N+2\gamma}  \diff \mu_k(y).
\end{align*}
We note that 
$$ \lambda_1(y) + \cdots + \lambda_N(y) = \text{Tr}(H(u)(y))
= \Delta u(y),$$
so using (\ref{uniformllyellipticoperator}) we have obtained
\begin{equation} \label{isofinalstep}
\begin{aligned}
\mu_k(B_1^\epsilon) 
& \leq \int_{\Gamma_u^\epsilon} \frac{1}{(N+2\gamma)^{N+2\gamma}} \left[w_k(y)^{-1} \nabla \cdot (w_k \nabla u)(y) \right]^{N+2\gamma} \diff \mu_k(y)
\\
& \leq  \left( \frac{1}{N+2\gamma} \cdot \frac{p(\Omega)}{\mu_k(\Omega)^{1-\frac{1}{N+2\gamma}}} \right)^{N+2\gamma}.
\end{aligned}
\end{equation} 

To conclude the proof, we note that the function $u(x)=\frac{1}{2}|x|^2$ solves the Neumann problem (\ref{neumannproblem}) in $B_1^\epsilon$ with constant $c=N+2\gamma$. Due to our observation that the Neumann problem has a solution if and only if $c=\frac{p(\Omega)}{w_k(\Omega)}$, we have that
\begin{equation} \label{isoconstant} \frac{p(B_1^\epsilon)}{\mu_k(B_1^\epsilon)} = N+2\gamma.
\end{equation}
Thus, (\ref{isofinalstep}) implies that
$$ \frac{p(B_1^\epsilon)}{\mu_k(B_1^\epsilon)^{1-\frac{1}{N+2\gamma}}} 
\leq \frac{p(\Omega)}{\mu_k(\Omega)^{1-\frac{1}{N+2\gamma}}},$$
which is what we wanted to prove.

Before we consider the case of general domain $\Omega$, we make some observations about the constant that appears in this inequality. The reflection group associated to the root system $R$ acts transitively on the set of Weyl chambers, i.e., for any $H,H'$ Weyl chambers, there exists $g\in G$ such that $gH=H'$. With a change of variables $gx=y$, this shows that
$$ p(H\cap B_1)=p(H' \cap B_1).$$
This, together with equation (\ref{isoconstant}), shows that the constant 
$$ \frac{p(H \cap B_1)}{\mu_k(H \cap B_1)^{1-\frac{1}{N+2\gamma}}}$$
does not depend on the choice of the Weyl chamber $H$.

For general $\Omega$, let $\Omega_H=\Omega\cap H$ for each Weyl chamber $H$. The subscript $H$ in the sums below will mean that the sum is taken over all Weyl chambers $H$. Since the sets $\Omega_H$ intersect only along the boundary, and $\overline{\Omega}=\displaystyle\bigcup_{H} \overline{\Omega_H}$, it is clear that  
$$ \mu_k(\Omega) = \sum_H \mu_k(\Omega_H).$$
Moreover, because $w_k(x)=0$ for any $x\in \partial\Omega_H \setminus \partial\Omega$, then we also have
$$ p(\Omega) = \sum_H p(\Omega_H).$$
Consequently, we have
\begin{equation} \label{isogencase}
\frac{p(\Omega)}{\mu_k(\Omega)^{1-\frac{1}{N+2\gamma}}} 
\geq \min_H \frac{p(\Omega_H)}{\mu_k(\Omega_H)^{1-\frac{1}{N+2\gamma}}}. 
\end{equation}
Indeed, if this were not true, then for any Weyl chamber $H'$ for which $p(\Omega_{H'})\neq 0$ and $\mu_k(\Omega_{H'})\neq 0$,  we have
$$ \frac{\displaystyle\sum_H p(\Omega_H)}{p(\Omega_{H'})} < \left( \frac{\displaystyle\sum_H \mu_k(\Omega_H)}{\mu_k(\Omega_{H'})} \right)^{1-\frac{1}{N+2\gamma}}
\leq \frac{\displaystyle\sum_H \mu_k(\Omega_H)}{\mu_k(\Omega_{H'})}.$$
By taking the power $-1$ and summing over all the Weyl chambers $H'$, we obtain a contradiction. Therefore, it is enough to prove the inequality for $\Omega_H$ where $H$ realises the minimum of the expression in (\ref{isogencase}), which was done in the first part of the proof.
\end{proof}

In the next step we will obtain a rearrangement inequality as a consequence of the isoperimetric inequality discussed above. Before this, we need to introduce some defintions.

Fix a Weyl chamber $\RR^N_\epsilon$. For any measurable subset $\Omega \subset \overline{\RR^N_\epsilon}$, we define its rearrangement to be the set
$$ \Omega^* = B_r(0) \cap \RR^N_\epsilon,$$
where $B_r(0)$ is the ball of radius $r$ centred at the origin, and $r\geq 0$ is such that 
$$ \mu_k(\Omega^*) = \mu_k(\Omega).$$

For any measurable function $f:\overline{\RR^N_\epsilon} \to \CC$, we define its symmetric decreasing rearrangement $f^* : \overline{\RR^N_\epsilon} \to [0,\infty)$ by
$$ f^*(x) = \int_0^\infty \1_{\{ |f| > t\}^*}(x) \diff t.$$
The function $f^*$ is radial and decreasing, and it satisfies the property
$$ \mu_k(\{x \in \RR^N_\epsilon : |f(x)| >t \}) = \mu_k(\{x \in \RR^N_\epsilon : f^*(x)>t\}),$$
for every $t>0$. As a consequence, we have
$$ \int_{\RR^N_\epsilon} |f|^p \diff \mu_k = \int_{\RR^N_\epsilon} |f^*|^p \diff \mu_k,$$
for any $p\geq 1$. 

The main result of Talenti \cite{T2} then implies the following rearrangement inequality.

%%%%%%%%%%%%%%%%%%%%%%%%%%%%%%%%%%%%%%%%%%%%%%%%%%%%%%%%%%%%%%%%%%%%%%%%%%%%%%%%%%%%%%%%%%%%%%%%%
\begin{prop} \label{euclideanrearrangement}
Let $1\leq p <\infty$, and let $H=\RR^N_\epsilon$ be a Weyl chamber. For any $f \in W^{1,p}(H)$, we have the inequality
$$ \int_{\RR^N_\epsilon} |\nabla f^*|^p \diff\mu_k \leq \int_{\RR^N_\epsilon} |\nabla f|^p \diff\mu_k .$$
\end{prop}
%%%%%%%%%%%%%%%%%%%%%%%%%%%%%%%%%%%%%%%%%%%%%%%%%%%%%%%%%%%%%%%%%%%%%%%%%%%%%%%%%%%%%%%%%%%%%%%%%

Indeed, the result in \cite{T2} proves a general rearrangement inequality on a subset of the Euclidean space with a weighted measure, as long as an isoperimetric inequality holds. Moreover, when equality in the isoperimetric inequality is given by balls, then the rearrangement function is radial.

Another consequence of Theorem \ref{iso} is the Sobolev inequality for the usual gradient $\nabla$ on the space weighted space $L^2(\mu_k)$. It is a classical fact the sharp isoperimetric inequality is equivalent to the sharp Sobolev inequality for $\nabla f\in L^1$. The proof of this equivalence is done through the co-area formula, see \cite{Mazya} for details. The general case, with $\nabla f \in L^p$, follows by taking $|f|^a$, for a suitable $a$, in the $L^1$ result.

%%%%%%%%%%%%%%%%%%%%%%%%%%%%%%%%%%%%%%%%%%%%%%%%%%%%%%%%%%%%%%%%%%%%%%%%%%%%%%%%%%%%%%%%%%%%%%%%%
\begin{prop} \label{classicalsobolevWeyl}
Let $\RR^N_\epsilon$ be a Weyl chamber. Let $1 \leq p < N+2\gamma$ and $q=\frac{p(N+2\gamma)}{N+2\gamma-p}$. Then there exists a constant $C>0$ such that for any $f\in C^1_c(\overline{\RR^N_\epsilon})$ we have
\begin{equation} \label{classicalsobolevWeylineq}
\left( \int_{\RR^N_\epsilon} |f|^q \diff\mu_k \right)^{1/q} \leq C \left( \int_{\RR^N_\epsilon} |\nabla f|^p \diff\mu_k \right)^{1/p}.
\end{equation}
\end{prop}
%%%%%%%%%%%%%%%%%%%%%%%%%%%%%%%%%%%%%%%%%%%%%%%%%%%%%%%%%%%%%%%%%%%%%%%%%%%%%%%%%%%%%%%%%%%%%%%%%

\subsection{Sharp constants}

Before we prove the main result, we need some information about the best constants that appear in the Sobolev inequality from Proposition \ref{classicalsobolevWeyl}, so we first compute
$$ C_W := \sup \frac{\left( \int_{\RR^N_\epsilon} |f|^q \diff\mu_k \right)^{1/q}}{\left( \int_{\RR^N_\epsilon} |\nabla f|^p \diff\mu_k \right)^{1/p}}.$$
The choice of Weyl chamber $\RR^N_\epsilon$ in this definition does not matter because we can obtain any Weyl chamber from another one through an element $g$ of the reflection group $G$. Firstly, it is clear that it is enough to consider only non-negative functions $f$ because replacing $f$ by $|f|$ leaves the quotient in the definition of $C_W$ invariant. By Proposition \ref{euclideanrearrangement}, since $\int_{\RR^N_\epsilon} f^q \diff\mu_k = \int_{\RR^N_\epsilon} (f^*)^q \diff\mu_k$, we have 
$$ \frac{\left( \int_{\RR^N_\epsilon} |f|^q \diff\mu_k \right)^{1/q}}{\left( \int_{\RR^N_\epsilon} |\nabla f|^p \diff\mu_k \right)^{1/p}}
\leq \frac{\left( \int_{\RR^N_\epsilon} |f^*|^q \diff\mu_k \right)^{1/q}}{\left( \int_{\RR^N_\epsilon} |\nabla f^*|^p \diff\mu_k \right)^{1/p}}.$$

But for $f$ radial, say $f(x)=g(|x|)$, where $g: \RR_+ \to \RR_+$, we have
$$ \nabla f(x) = \frac{x}{|x|} g'(|x|),$$
so, using polar coordinates, we can compute
\begin{equation} \label{bestconstantclassical1} 
\frac{\left( \int_{\RR^N_\epsilon} |f|^q \diff\mu_k \right)^{1/q}}{\left( \int_{\RR^N_\epsilon} |\nabla f|^p \diff\mu_k \right)^{1/p}}
= p(B_1^\epsilon)^{\frac{1}{q}-\frac{1}{p}} \frac{\left(\int_0^\infty g(r)^q r^{N+2\gamma-1} \diff r \right)^{1/q}}{\left(\int_0^\infty |g'(r)|^p r^{N+2\gamma-1} \diff r\right)^{1/p}}.
\end{equation}
We have thus reduced the problem to maximising the functional 
\begin{equation} \label{functionalJ} 
J(g) = \frac{\left(\displaystyle\int_0^\infty g(r)^q r^{N+2\gamma-1} \diff r \right)^{1/q}}{\left( \displaystyle\int_0^\infty |g'(r)|^p r^{N+2\gamma-1} \diff r \right)^{1/p}}.
\end{equation}
But this problem was studied by Talenti in \cite{T}, where he obtained that 
\begin{equation} \label{bestconstantclassical2} 
J(g) \leq (N+2\gamma)^{-1/p} \left(\frac{p-1}{N+2\gamma - p} \right)^{1/p'} \left[ \frac{1}{p'} B\left( \frac{N+2\gamma}{p}, \frac{N+2\gamma}{p'}\right) \right]^{-\frac{1}{N+2\gamma}},
\end{equation}
where $B$ is the beta function and $p'$ is the conjugate of $p$, i.e., $\frac{1}{p}+\frac{1}{p'}=1$. Equality is achieved for functions of the form
\begin{equation} \label{equalitycase} 
\varphi(r) = (a+br^{p'})^{1-\frac{N+2\gamma}{p}},
\end{equation}
where $a,b>0$. 

By considering polar coordinates in the definition of the Macdonald-Mehta constant $M_k$, it can be seen that
$$ p(B_1) = \frac{M_k}{2^{\frac{N+2\gamma}{2}-1}\Gamma\left(\frac{N+2\gamma}{2} \right)}.$$
Moreover, by a change of variables we can see that $p(B_1^\epsilon)$ does not depend on the Weyl chamber, and since the number of Weyl chambers is equal to the order of the reflection group $G$, it follows that
$$ p(B_1^\epsilon) = \frac{1}{|G|} p(B_1).$$
Therefore, combining (\ref{bestconstantclassical1}) and (\ref{bestconstantclassical2}), we obtain the value of the best constant in the Sobolev inequality (\ref{classicalsobolevWeylineq}); this is summarised in the following result. 

%%%%%%%%%%%%%%%%%%%%%%%%%%%%%%%%%%%%%%%%%%%%%%%%%%%%%%%%%%%%%%%%%%%%%%%%%%%%%%%%%%%%%%%%%%%%%%%%%
\begin{prop}
Let $\RR^N_\epsilon$ be a Weyl chamber. Let $1<p<N+2\gamma$, and $q=\frac{p(N+2\gamma)}{N+2\gamma-p}$. Then the best constant in the Sobolev inequality (\ref{classicalsobolevWeylineq}) is given by
\begin{align*}
C_W 
&= (N+2\gamma)^{-1/p} \left(\frac{p-1}{N+2\gamma - p} \right)^{1/p'} 
\\
&\qquad \qquad \qquad 
\cdot \left[ \frac{2^{\frac{N+2\gamma}{2}-1}p'|G|}{M_k}\frac{\Gamma(N+2\gamma)\Gamma((N+2\gamma)/2)}{\Gamma((N+2\gamma)/p)\Gamma((N+2\gamma)/p')}\right]^{\frac{1}{N+2\gamma}}.
\end{align*}
\end{prop}
%%%%%%%%%%%%%%%%%%%%%%%%%%%%%%%%%%%%%%%%%%%%%%%%%%%%%%%%%%%%%%%%%%%%%%%%%%%%%%%%%%%%%%%%%%%%%%%%%

\begin{remark}
This generalises the results of \cite{CRO} where the authors consider the weighted Sobolev inequality with monomial weight $|x_1|^{A_1} \cdot \ldots \cdot |x_l|^{A_l}$; this corresponds to the Dunkl weight with root system $R=\{e_1, \ldots, e_l\}$, where $1\leq l \leq N$ and $e_1, \ldots, e_N$ is the standard basis of $\RR^N$.
\end{remark}

We now turn to the Sobolev inequality on the whole space $\RR^N$, which follows from the above.

%%%%%%%%%%%%%%%%%%%%%%%%%%%%%%%%%%%%%%%%%%%%%%%%%%%%%%%%%%%%%%%%%%%%%%%%%%%%%%%%%%%%%%%%%%%%%%%%%
\begin{prop}
Let $1<p<N+2\gamma$ and $g=\frac{p(N+2\gamma)}{N+2\gamma-p}$. Then there exists a constant $C_{CS}>0$ such that for any $f\in C_c^\infty(\RR^N)$, we have the inequality
\begin{equation} \label{classicalsobolevineqRR^N}
\norm{f}_q \leq C_{CS} \norm{\nabla f}_p.
\end{equation}
Moreover, the sharp constant $C_{CS}$ satisfies
$$ C_{CS} = C_W,$$
with equality if and only if $f$ is supported on the closure of a Weyl chamber, where it takes the form (\ref{equalitycase}).
\end{prop}
%%%%%%%%%%%%%%%%%%%%%%%%%%%%%%%%%%%%%%%%%%%%%%%%%%%%%%%%%%%%%%%%%%%%%%%%%%%%%%%%%%%%%%%%%%%%%%%%%

\begin{proof}
Fix $f\in C_c^\infty(\RR^N)$ and for any Weyl chamber $H$ let $\restr{f}{H}$ denote the restriction of $f$ to $H$. From the Sobolev inequality applied on each Weyl chamber, we have
\begin{align*}
\norm{f}_q 
= \left(\sum_H \norm{\restr{f}{H}}_q^q \right)^{1/q} 
\leq C_W \left(\sum_H \norm{\nabla \restr{f}{H}}_p^q \right)^{1/q}.
\end{align*}
Now we apply the inequality $\norm{y}_{l_q} \leq \norm{y}_{l_p}$ which holds on finite dimensional spaces since $p\leq q$, with equality if and only if $y=(0,\ldots, 0, 1, 0, \ldots, 0)$. It follows that
$$ \norm{f}_q \leq C_W \left(\sum_H \norm{\nabla \restr{f}{H}}_p^p \right)^{1/p}
=C_W \norm{\nabla f}_p.$$
This proves that the Sobolev inequality (\ref{classicalsobolevineqRR^N}) holds, and also that $C_{CS} \leq C_W$. Moreover, following the proof backwards we see that equality does hold if and only if $f$ is supported on $\overline{H}$ for some Weyl chamber $H$, and we have equality in the corresponding Weyl chamber inequality, i.e., 
$$f(x)=(a+b|x|^{p'})^{1-\frac{N+2\gamma}{p}} \qquad \forall x\in H,$$ 
for some $a,b>0$. 
\end{proof}

\begin{remark}
In view of Lemma \ref{dunkl-euclidean-dirichletform}, this result for $p=2$ provides yet another proof of Theorem \ref{sobolev}.  
\end{remark}

We now specialise to the case $p=2$. This limitation comes from the fact that we rely on the Dirichlet form method and Lemma \ref{dunkl-euclidean-dirichletform}. It remains an open question to prove the inequality 
$$ \norm{\nabla f}_p \leq \norm{\nabla_k f}_p$$
for $p\neq 2$. If this holds, then the estimates below generalise immediately. 

We are now ready to prove the main result.

\begin{proof} [Proof of Theorem \ref{sobolevconstantsthm}]

The upper bound follows immediately from the above, using Lemma \ref{dunkl-euclidean-dirichletform}. Indeed, we have
$$ \frac{\norm{f}_q}{\norm{\nabla_k f}_2} \leq \frac{\norm{f}_q}{\norm{\nabla f}_2},$$
and taking supremum on both sides we obtain $C_{DS} \leq S_{SC}$. In the case $p=2$ the value of $C_W$ simplifies and we have indeed
$$ C_{CS} = C_W 
=\left(\frac{2}{(N+2\gamma)(N+2\gamma-2)} \right)^{1/2} 
\cdot \left[ \frac{|G|}{M_k}\frac{\Gamma(N+2\gamma)}{\Gamma((N+2\gamma)/2)}\right]^{\frac{1}{N+2\gamma}}.$$

For the lower bound, we note that
$$ \displaystyle\sup_{f \text{ radial}} \frac{\norm{f}_q}{\norm{\nabla_k f}_2} \leq C_{DS}.$$
The left hand side can be computed as above, making use of the functional (\ref{functionalJ}) and Talenti's result. We obtain
\begin{align*}
\displaystyle\sup_{f \text{ radial}} \frac{\norm{f}_q}{\norm{\nabla_k f}_2}
= \left(\frac{2}{(N+2\gamma)(N+2\gamma-2)} \right)^{1/2} 
\cdot \left[ \frac{1}{M_k}\frac{\Gamma(N+2\gamma)}{\Gamma((N+2\gamma)/2)}\right]^{\frac{1}{N+2\gamma}}.
\end{align*}
This completes the proof.
\end{proof}

\noindent \textbf{Acknowledgements.} The author wishes to thank Charles F Dunkl for pointing out useful references. 
%%%%%%%%%%%%%%%%%%%%%%%%%%%%%%%%%%%%%%%%%%%%%%%%%%%%%%%%%%%%%%%%%
%\pagebreak
%\nocite{*}

\bibliographystyle{plain}
\bibliography{ref}

\begin{thebibliography}{10}

\bibitem{AS}
B.~Amri and M.~Sifi.
\newblock Riesz transforms for {D}unkl transform.
\newblock {\em Ann. Math. Blaise Pascal}, 19(1):247--262, 2012.

\bibitem{Anker}
J.-P. Anker.
\newblock An introduction to {D}unkl theory and its analytic aspects.
\newblock In {\em Analytic, algebraic and geometric aspects of differential
  equations}, Trends Math., pages 3--58. Birkh\"{a}user/Springer, Cham, 2017.

\bibitem{A}
T.~Aubin.
\newblock Meilleures constantes dans le th\'{e}or\`eme d'inclusion de {S}obolev
  et un th\'{e}or\`eme de {F}redholm non lin\'{e}aire pour la transformation
  conforme de la courbure scalaire.
\newblock {\em J. Funct. Anal.}, 32(2):148--174, 1979.

\bibitem{BCLS}
D.~Bakry, T.~Coulhon, M.~Ledoux, and L.~Saloff-Coste.
\newblock Sobolev inequalities in disguise.
\newblock {\em Indiana Univ. Math. J.}, 44(4):1033--1074, 1995.

\bibitem{BEHL}
A.~Balinsky, W.~D. Evans, D.~Hundertmark, and R.~T. Lewis.
\newblock On inequalities of {H}ardy-{S}obolev type.
\newblock {\em Banach J. Math. Anal.}, 2(2):94--106, 2008.

\bibitem{CRO}
X.~Cabr\'{e} and X.~Ros-Oton.
\newblock Sobolev and isoperimetric inequalities with monomial weights.
\newblock {\em J. Differential Equations}, 255(11):4312--4336, 2013.

\bibitem{CDDDV}
A.~Cohen, W.~Dahmen, I.~Daubechies, and R.~DeVore.
\newblock Harmonic analysis of the space {BV}.
\newblock {\em Rev. Mat. Iberoamericana}, 19(1):235--263, 2003.

\bibitem{CDVPX}
A.~Cohen, R.~DeVore, P.~Petrushev, and H.~Xu.
\newblock Nonlinear approximation and the space {${\rm BV}({\bf R}^2)$}.
\newblock {\em Amer. J. Math.}, 121(3):587--628, 1999.

\bibitem{Davies}
E.~B. Davies.
\newblock {\em Heat kernels and spectral theory}, volume~92 of {\em Cambridge
  Tracts in Mathematics}.
\newblock Cambridge University Press, Cambridge, 1989.

\bibitem{D}
P.~Doktor.
\newblock Approximation of domains with {L}ipschitzian boundary.
\newblock {\em \v{C}asopis P\v{e}st. Mat.}, 101(3):237--255, 1976.

\bibitem{E}
P.~Etingof.
\newblock A uniform proof of the {M}acdonald-{M}ehta-{O}pdam identity for
  finite {C}oxeter groups.
\newblock {\em Math. Res. Lett.}, 17(2):275--282, 2010.

\bibitem{HMS}
S.~Hassani, S.~Mustapha, and M.~Sifi.
\newblock Riesz potentials and fractional maximal function for the {D}unkl
  transform.
\newblock {\em J. Lie Theory}, 19(4):725--734, 2009.

\bibitem{H}
J.~E. Humphreys.
\newblock {\em Reflection groups and {C}oxeter groups}, volume~29 of {\em
  Cambridge Studies in Advanced Mathematics}.
\newblock Cambridge University Press, Cambridge, 1990.

\bibitem{L}
M.~Ledoux.
\newblock On improved {S}obolev embedding theorems.
\newblock {\em Math. Res. Lett.}, 10(5-6):659--669, 2003.

\bibitem{LL}
E.~H. Lieb and M.l Loss.
\newblock {\em Analysis}, volume~14 of {\em Graduate Studies in Mathematics}.
\newblock American Mathematical Society, Providence, RI, second edition, 2001.

\bibitem{M}
I.~G. Macdonald.
\newblock Some conjectures for root systems.
\newblock {\em SIAM J. Math. Anal.}, 13(6):988--1007, 1982.

\bibitem{Mazya}
V.~Maz'ya.
\newblock {\em Sobolev spaces with applications to elliptic partial
  differential equations}, volume 342 of {\em Grundlehren der Mathematischen
  Wissenschaften [Fundamental Principles of Mathematical Sciences]}.
\newblock Springer, Heidelberg, augmented edition, 2011.

\bibitem{N}
J.~Nash.
\newblock Continuity of solutions of parabolic and elliptic equations.
\newblock {\em Amer. J. Math.}, 80:931--954, 1958.

\bibitem{O}
E.~M. Opdam.
\newblock Some applications of hypergeometric shift operators.
\newblock {\em Invent. Math.}, 98(1):1--18, 1989.

\bibitem{Rosler}
M.~R\"{o}sler.
\newblock Dunkl operators: theory and applications.
\newblock In {\em Orthogonal polynomials and special functions ({L}euven,
  2002)}, volume 1817 of {\em Lecture Notes in Math.}, pages 93--135. Springer,
  Berlin, 2003.

\bibitem{SC}
L.~Saloff-Coste.
\newblock {\em Aspects of {S}obolev-type inequalities}, volume 289 of {\em
  London Mathematical Society Lecture Note Series}.
\newblock Cambridge University Press, Cambridge, 2002.

\bibitem{SC2}
L.~Saloff-Coste.
\newblock Pseudo-{P}oincar\'{e} inequalities and applications to {S}obolev
  inequalities.
\newblock In {\em Around the research of {V}ladimir {M}az'ya. {I}}, volume~11
  of {\em Int. Math. Ser. (N. Y.)}, pages 349--372. Springer, New York, 2010.

\bibitem{T}
G.~Talenti.
\newblock Best constant in {S}obolev inequality.
\newblock {\em Ann. Mat. Pura Appl. (4)}, 110:353--372, 1976.

\bibitem{T2}
G.~Talenti.
\newblock A weighted version of a rearrangement inequality.
\newblock {\em Ann. Univ. Ferrara Sez. VII (N.S.)}, 43:121--133 (1998), 1997.

\bibitem{Tri}
H.~Triebel.
\newblock {\em Theory of function spaces. {II}}, volume~84 of {\em Monographs
  in Mathematics}.
\newblock Birkh\"{a}user Verlag, Basel, 1992.

\end{thebibliography}

\end{document}